\documentclass[12pt]{amsart}
\usepackage[
  dvipdfmx,
  colorlinks=true,
  linkcolor=black,
  citecolor=black,
  urlcolor=blue
]{hyperref}
\usepackage{color}
\usepackage{amsmath,amssymb,amsthm}
\usepackage{amscd}
\usepackage{comment}
\usepackage{overpic}
\usepackage{tabularx}
\usepackage[dvipdfmx,left=25truemm,right=25truemm,top=30truemm,bottom=30truemm]{geometry}
\usepackage{bm}

\newtheorem{theorem}{Theorem}[section]

\newtheorem{proposition}[theorem]{Proposition}
\newtheorem{corollary}[theorem]{Corollary}

\theoremstyle{definition} 
\newtheorem{remark}[theorem]{Remark}

\newcommand{\textR}[1]{#1}

\newcommand{\R}{\mathbb{R}}

\renewcommand{\epsilon}{\varepsilon}

\renewcommand{\geq}{\geqslant}

\makeatletter
\@addtoreset{equation}{section}

\makeatother
\begin{document}

\title[A Monge normal form of a surface in $\R^4$ and its applications]
{A curvature-ellipse-based Monge normal form of a regular surface in $\R^4$ and its applications}

\author[M.~Hasegawa]
{Masaru Hasegawa}

\address[Masaru Hasegawa]{%
Department of Information Science, Center for Liberal Arts and Sciences, Iwate Medical University,
1-1-1 Idaidori, Yahaba-cho, Shiwa-gun, Iwate 028-3694, Japan.}
\email{mhase@iwate-med.ac.jp}


%\date{\today}
%
%\thanks{This work is partially supported by JSPS KAKENHI Grant Numbers 21K03230 and 22K03312}

%%%%%%%%%%%%%%%%%%%%%%%%%%%%%%%%%%%%%%%%%%%%%%%%%%

\allowdisplaybreaks[4]
%\setlength{\abovedisplayskip}{6pt}
%\setlength{\belowdisplayskip}{6pt}
%\renewcommand{\arraystretch}{1.3}
%\renewcommand{\baselinestretch}{1.2}

%%%%%%%%%%%%%%%%%%%%%%%%%%%%%%%%%%%%%%%%%%%%%%%%%%

\maketitle

%%%%%%%%%%%%%%%%%%%%%%%%%%%%%%%%%%%%%%%%%%%%%%%%%%

\begin{abstract}
We study a Monge normal form of a surface in $\R^4$ based on the principal axes of its curvature ellipse.
The curvature-ellipse-based Monge normal form is obtained using only Euclidean motions of the ambient space and local reparameterizations of the surface, and preserves the Euclidean invariants of the curvature ellipse, including its semi-axis lengths up to their ordering.
We also give applications of this normal form to the $2$-jet geometry of
surfaces in $\R^4$ and to the geometry of their orthogonal projections.
\end{abstract}

%%%%%%%%%%%%%%%%%%%%%%%%%%%%%%%%%%%%%%%%%%%%%%%

\section{Introduction}
\label{sec:introduction}

The curvature ellipse has a long history in the differential geometry of surfaces in $\R^4$. 
It was studied in the classical works of Moore and Wilson \cite{MW1916}, Wong \cite{Wong1946}, and Little \cite{Little1969}.
The relation between the curvature ellipse and the normal curvature was studied by Asperti \cite{Asperti1980} and Guadalupe and Rodriguez \cite{GR1983}, also in the more general setting of space forms.

The geometry of surfaces in $\R^4$ has also been studied from the viewpoint of contact and singularity theory.
Mochida, Romero-Fuster and Ruas \cite{MRFR1995} studied contact with hyperplanes using height functions, whereas Romero-Fuster and Sanabria-Codesal \cite{RFSC2013} studied contact with hyperspheres using distance-squared functions.
The geometry of asymptotic directions and inflection points was studied by Garcia, Mochida, Romero-Fuster and Ruas \cite{GMRFR2000} and Bruce and Tari \cite{BT2002} using singularities of height functions.

Orthogonal projections of surfaces in $\R^4$ into 3-space have also been studied from the viewpoint of singularity theory.
Mond \cite{Mond1982} related asymptotic tangent directions to singularities of tangent projections.
This direction was further developed by Bruce and Nogueira \cite{BN1998}, Nu\~no-Ballesteros and Tari \cite{NBT2007}, Oset Sinha and Tari \cite{OST2015}, and Benedini Riul and Oset Sinha \cite{BROS2019}.

Other studies on the geometry of surfaces in $\R^4$ include, for example, those in \cite{Asperti1984, AFR1982, BSB2012, GS2000, GGTG1997, Maccori2004}.

To investigate the geometry of surfaces in $\R^4$, it is important to preserve the Euclidean metric information contained in the second fundamental form. 
We therefore seek a normalization of the second-order Monge normal form which uses only Euclidean motions of the ambient space and local reparameterizations of the surface. 
The principal axes of the curvature ellipse provide such a normalization. 

The significance of Theorem~\ref{thm:Euclidean_Monge_normal_form} for the study of the geometry of the surface, for instance in contact problems with model submanifolds, is that this normalization preserves the Euclidean geometry of the curvature ellipse, in particular, its center and its semi-axis lengths up to their ordering. 
The result most closely related to Theorem~\ref{thm:Euclidean_Monge_normal_form} is that of Bayard and S\'anchez-Bringas~\cite{BSB2012}. 
They introduced four closely related quantities $a$, $b$, $\alpha$, $\beta$, whose squares classify the associated quadratic map up to additional reflection ambiguities. 
Although their result is closely related to the second-order data in Theorem~\ref{thm:Euclidean_Monge_normal_form}, it does not directly provide the Monge parameterization given here.

The essential quantities are therefore the Euclidean invariants of the curvature ellipse, rather than the particular choice or ordering of a principal-axis frame. 
This holds true even at points where the curvature ellipse is degenerate: although an adapted frame may not be unique, the unordered pair of semi-axis lengths $\{\lambda,\mu\}$ is unchanged.

For this reason, we use the above orthogonal normal form rather than normal forms obtained under a general $GL(2,\R)\times GL(2,\R)$ action on pairs of quadratic forms. 
Although such normal forms are useful for affine or flat-geometric classifications, they do not in general preserve the Euclidean curvature ellipse. 
The normal form given by Theorem~\ref{thm:Euclidean_Monge_normal_form} is therefore appropriate for the geometric applications considered here.

The paper is organized as follows: In Section~\ref{sec:Monge_form}, we introduce a curvature-ellipse-based Monge normal form of a surface in $\R^4$. 
In Section~\ref{sec:2-jet_geometry}, we describe the 2-jet geometry of surfaces in $\R^4$ in terms of the coefficients of the curvature-ellipse-based Monge normal form. 
In Section~\ref{sec:projection}, using this normal form, we investigate
the relation between the geometry of surfaces in $\R^4$ and that of their orthogonal projections into the subspaces perpendicular to the projection directions.
 
%%%%%%%%%%%%%%%%%%%%%%%%%%%%%%%%%%%

\section{A curvature-ellipse-based Monge normal form}
\label{sec:Monge_form}

The second-order extrinsic geometry of a regular surface in $\R^4$ is determined by its vector-valued second fundamental form. 
Let $U$ be an open subset of $\R^2$ containing the origin, and let
$$
X\colon U \to \R^4
$$
be a local parameterization of a regular surface $M \subset \R^4$ such that $X(0,0)=\bm{p}\in M$. 
Then, $\{X_u(0,0),X_v(0,0)\}$ is a tangent frame of $T_{\bm{p}}M$. 
We also take an orthonormal frame $\{\bm{\nu}_1,\bm{\nu}_2\}$ of $N_{\bm{p}}M$.
We denote the \textit{first fundamental form at $\bm{p}$} by
$$
I_{\bm{p}}\colon T_{\bm{p}}M \times T_{\bm{p}}M \to \R, 
$$
and $I_{\bm{p}}$ is given by
$$
I_{\bm{p}} = \langle dX, dX\rangle = E du^2+ 2F du dv+G dv^2,
$$
where
$$
E=\langle X_u, X_u\rangle, \quad F=\langle X_u, X_v\rangle, \quad G=\langle X_v, X_v \rangle.
$$
We denote by
$$
II_{\bm{p}}\colon T_{\bm{p}}M \times T_{\bm{p}}M \to N_{\bm{p}}M
$$
the \textit{second fundamental form at $\bm{p}$}.
The second fundamental form $II_{\bm{p}}$ is given by
$$
II_{\bm p} = \langle\bm\nu_1,d^2 X\rangle\bm\nu_1 + \langle\bm\nu_2,d^2 X\rangle\bm\nu_2,
$$
where
$$
\langle\bm\nu_i,d^2 X\rangle = l_i\,du^2+2m_i\,du\,dv+n_i\,dv^2, \qquad i=1,2.
$$
If $X$ is given by a Monge form $X(u,v)=(u,v,f(u,v),g(u,v))$ with $f$ and $g$ having no constant or linear parts, we obtain 
\begin{align*}
l_1 & = f_{uu}(0,0), \quad m_1=f_{uv}(0,0), \quad n_1=f_{vv}(0,0), \\
l_2 & = g_{uu}(0,0),\quad m_2=g_{uv}(0,0), \quad n_2=g_{vv}(0,0).
\end{align*}

The \textit{normal curvature vector} $\kappa_n$ at $\bm{p}$ in a tangent direction $\bm{w} \in T_{\bm{p}}M\setminus\{\bm{0}\}$ is defined by
$$
\kappa_n(\bm{p},\bm{w}) = \dfrac{II_{\bm{p}}(\bm{w},\bm{w})}{I_{\bm{p}}(\bm{w},\bm{w})}.
$$
We denote by
$$
\kappa_n(\bm{p})\colon S(T_{\bm{p}}M)\to N_{\bm{p}}M
$$
its restriction to the unit circle.
The image of the unit circle in $T_{\bm{p}}M$ by $\kappa_n(\bm{p})$, being a quadratic map, is an ellipse, possibly degenerating to a point or a line segment, in the normal plane $N_{\bm{p}}M$, called the \textit{curvature ellipse at $\bm{p}$} and denoted by $\mathcal{E}_{\bm{p}}$. 
Therefore, 
\begin{equation*}
%\label{eq:def_curvature_ellipse}
\mathcal{E}_{\bm{p}}=\{\kappa_n(\bm{p},\bm{w}): \bm{w}\in T_{\bm{p}}M, I_{\bm{p}}(\bm{w},\bm{w})=1\} \subset N_{\bm{p}}M.
\end{equation*}
The center of $\mathcal{E}_{\bm{p}}$ is the \textit{mean curvature vector $\bm{H}$ at $\bm{p}$}, which is given by
$$
\bm{H}=h_1\bm{\nu}_1+h_2\bm{\nu}_2=\dfrac{En_1-2Fm_1+Gl_1}{2(EG-F^2)}\bm{\nu}_1 + \dfrac{En_2-2Fm_2+Gl_2}{2(EG-F^2)}\bm{\nu}_2.
$$
We denote the lengths of the two semi-axes of $\mathcal{E}_{\bm{p}}$ by $\lambda$, $\mu\ge 0$.

Constructions of frames adapted to the curvature ellipse have already
appeared in the literature. 
Guadalupe and Rodriguez \cite{GR1983} described the principal-axis construction of the curvature ellipse and also considered its degenerate cases. 
Asperti, Ferus and Rodriguez \cite{AFR1982} constructed positively oriented orthonormal tangent and normal frames in terms of the trace-free parts of the shape operators. 
Asperti \cite{Asperti1984} expressed this construction in terms of the principal axes of the curvature ellipse in the non-degenerate case (see \cite{GR1983}). 
Maccori \cite{Maccori2004} obtained a closely related normal form for the second fundamental form and related its four coefficients to the center and the semi-axes of the curvature ellipse. 

More directly related to Theorem~\ref{thm:Euclidean_Monge_normal_form}, Bayard and S\'anchez-Bringas~\cite{BSB2012} introduced four quantities $a$, $b$ $,\alpha$, $\beta$, whose squares classify the associated quadratic map up to additional reflection ambiguities.

The following theorem gives the corresponding Monge normal form, including the cases where the curvature ellipse degenerates.

\begin{theorem}
\label{thm:Euclidean_Monge_normal_form}
Let $M$ be a regular surface in $\R^4$, and let $\bm{p}\in M$.
After a Euclidean motion of $\R^4$ and a local reparameterization of $M$, there exists a local Monge normal form at $\bm{p}$ such that
\begin{equation*}
%\label{eq:normal_form}
X(u,v)=\left(u,v,f(u,v),g(u,v)\right),
\end{equation*}
where
\begin{align*}
f(u,v) &= \frac{1}{2}\left\{h_1(u^2+v^2)+\lambda(u^2-v^2) \right\}+\sum_{n\ge3}f_n(u,v),\quad f_n(u,v)=\sum_{i+j=n}\dfrac1{i!j!}a_{ij}u^iv^j, \\
g(u,v) &= \frac{1}{2} \left\{h_2(u^2+v^2)+2\mu uv\right\}+\sum_{n\ge3}g_n(u,v),\quad g_n(u,v)=\sum_{i+j=n}\dfrac1{i!j!}b_{ij}u^iv^j.
\end{align*}
\end{theorem}

\begin{remark}
The two principal axes of the curvature ellipse may be interchanged.
Therefore, although the respective roles of $\lambda$ and $\mu$ depend on the choice of the adapted frame, the unordered pair $\{\lambda,\mu\}$ is determined by the Euclidean geometry of the curvature ellipse. 
\end{remark}

\begin{proof}[Proof of Theorem~\ref{thm:Euclidean_Monge_normal_form}]
After a translation, we may assume that $\bm{p}=\bm{0}$.
Choose an orthonormal basis $\{\bm{e}_1,\bm{e}_2\}$ of $T_{\bm{p}}M$.

Since $\{\bm{e}_1,\bm{e}_2\}$ is an orthonormal basis of $T_{\bm{p}}M$, we have
$$
\bm{H} = \frac{1}{2}\left(II_{\bm{p}}(\bm{e}_1,\bm{e}_1)+II_{\bm{p}}(\bm{e}_2,\bm{e}_2)\right).
$$
Set
$$
\bm{B} = \frac{1}{2}\left(II_{\bm{p}}(\bm{e}_1,\bm{e}_1)-II_{\bm{p}}(\bm{e}_2,\bm{e}_2)\right), \quad
\bm{C}=II_{\bm{p}}(\bm{e}_1,\bm{e}_2).
$$
For a unit tangent vector $\bm{w}(\theta) = \cos\theta\,\bm{e}_1+\sin\theta\,\bm{e}_2$, the symmetry and bilinearity of $II_{\bm{p}}$ give
$$
II_{\bm{p}}(\bm{w}(\theta),\bm{w}(\theta)) = \cos^2\theta\,II_{\bm{p}}(\bm{e}_1,\bm{e}_1) +2\cos\theta\sin\theta\,II_{\bm{p}}(\bm{e}_1,\bm{e}_2)+\sin^2\theta\,II_{\bm{p}}(\bm{e}_2,\bm{e}_2).
$$
Therefore, by the definitions of $\bm{H}$, $\bm{B}$, and $\bm{C}$,
\begin{equation*}
%\label{eq:curvature-ellipse-parameterization}
II_{\bm{p}}(\bm{w}(\theta),\bm{w}(\theta)) = \bm{H}+\bm{B}\cos2\theta+\bm{C}\sin2\theta.
\end{equation*}
Since every unit tangent vector is of the form $\bm{w}(\theta)$, the
curvature ellipse is parametrized by
\begin{equation}
\label{eq:parameterization_of_curvature_ellipse}
\mathcal{E}_{\bm{p}}(\theta) = \bm{H}+\bm{B}\cos2\theta+\bm{C}\sin2\theta.
\end{equation}

We shall show that the tangent frame can be rotated so that the vectors multiplying $\cos2\theta$ and $\sin2\theta$ in the parameterization of $\mathcal{E}_{\bm{p}}(\theta)$ are orthogonal. 
Rotate the orthonormal tangent frame
through an angle $\alpha$:
$$
\bm{e}_1' = \cos\alpha\,\bm{e}_1+\sin\alpha\,\bm{e}_2, \quad
\bm{e}_2' = -\sin\alpha\,\bm{e}_1+\cos\alpha\,\bm{e}_2.
$$
With respect to the rotated tangent frame, set $\bm{w}(\theta)' = \cos\theta\,\bm{e}_1' + \sin\theta\,\bm{e}_2'$.
Then
\begin{align*}
\bm{w}(\theta)' &=\cos\theta\left(\cos\alpha\,\bm{e}_1+\sin\alpha\,\bm{e}_2\right)+ \sin\theta\left(-\sin\alpha\,\bm{e}_1+\cos\alpha\,\bm{e}_2\right) \\
&=\cos(\theta+\alpha)\,\bm{e}_1+\sin(\theta+\alpha)\,\bm{e}_2 \\
&= \bm{w}(\theta+\alpha).
\end{align*}
Hence, the corresponding parameterization of the curvature ellipse is
\begin{equation*}
\label{eq:new_parameterization_of_curvature_ellipse}
II_{\bm{p}}(\bm{w}(\theta)',\bm{w}(\theta)') = \mathcal{E}_{\bm{p}}(\theta+\alpha).
\end{equation*}
Using \eqref{eq:parameterization_of_curvature_ellipse}, we obtain
\begin{align*}
\mathcal{E}_{\bm{p}}(\theta+\alpha)&=\bm{H}+\bm{B}\cos2(\theta+\alpha)+\bm{C}\sin2(\theta+\alpha) \\
&=\bm{H}+\left(\bm{B}\cos2\alpha+\bm{C}\sin2\alpha\right)\cos2\theta+\left(-\bm{B}\sin2\alpha+\bm{C}\cos2\alpha\right)\sin2\theta.
\end{align*}
Therefore, setting
$$
\bm{B}'=\bm{B}\cos2\alpha+\bm{C}\sin2\alpha, \quad
\bm{C}'=-\bm{B}\sin2\alpha+\bm{C}\cos2\alpha,
$$
we have
$$
\mathcal{E}_{\bm{p}}(\theta+\alpha)=\bm{H}+\bm{B}'\cos2\theta+\bm{C}'\sin2\theta.
$$

A direct computation gives
$$
\langle \bm{B}',\bm{C}'\rangle=\frac{1}{2}\left(\lVert \bm{C}\rVert^2-\lVert \bm{B}\rVert^2\right)\sin4\alpha+\langle\bm{B},\bm{C}\rangle\cos4\alpha.
$$
Hence, an angle $\alpha$ can always be chosen so that $\langle \bm{B}',\bm{C}'\rangle=0$.
Thus $\bm{B}'$ and $\bm{C}'$ are directed along the principal axes of the
curvature ellipse.

Write $\lambda=\lVert \bm{B}'\rVert$, $\mu=\lVert \bm{C}'\rVert$.
As mentioned above, the non-negative numbers $\lambda$ and $\mu$ are the semi-axis lengths of the curvature ellipse. 
If the ellipse is degenerate, one or both of these numbers vanish; the corresponding principal-axis frame is then non-unique. 
Moreover, the two principal axes may be interchanged, so that the semi-axis lengths $\lambda$ and $\mu$ are Euclidean invariants up to their ordering.

Choose an orthonormal basis $\{\bm\nu_1,\bm\nu_2\}$ of $N_pM$ adapted to these principal axes, so that $\bm{B}'=\lambda\bm\nu_1$ and $\bm{C}'=\mu\bm\nu_2$.
When one of the semi-axis lengths vanishes, the corresponding normal direction may be chosen so that $\{\bm\nu_1,\bm\nu_2\}$ is an orthonormal frame.
If both semi-axis lengths vanish, we may choose any orthonormal frame $\{\bm\nu_1,\bm\nu_2\}$ of $N_{\bm{p}}M$.
Thus, in all cases, $\bm{B}'=\lambda\bm\nu_1$ and $\bm{C}'=\mu\bm\nu_2$ with $\lambda,\mu\geq0$.

We drop the primes from the rotated tangent frame and denote $\{\bm{e}_1',\bm{e}_2'\}$ again by $\{\bm{e}_1,\bm{e}_2\}$.
We have
$$
\bm{H}=h_1\bm\nu_1+h_2\bm\nu_2,
$$
and hence
\begin{equation}
\label{eq:curvature_ellipse_normal_form}
II_{\bm{p}}(\bm{w}(\theta),\bm{w}(\theta))=\mathcal{E}_{\bm{p}}(\theta)=\left(h_1+\lambda\cos2\theta\right)\bm\nu_1+\left(h_2+\mu\sin2\theta\right)\bm\nu_2.
\end{equation}
Applying the definitions of $\bm{H}$, $\bm{B}$, and $\bm{C}$ to the rotated tangent frame, with the corresponding vectors denoted by $\bm{B}'$ and $\bm{C}'$, we have
$$
II_{\bm{p}}(\bm{e}_1,\bm{e}_1)=\bm{H}+\bm{B}'\quad II_{\bm{p}}(\bm{e}_1,\bm{e}_2)=\bm{C}',\quad II_{\bm{p}}(\bm{e}_2,\bm{e}_2)=\bm{H}-\bm{B}'.
$$
Since
$$
\bm{H}=h_1\bm\nu_1+h_2\bm\nu_2,\quad \bm{B}'=\lambda\bm\nu_1,\quad \bm{C}'=\mu\bm\nu_2,
$$
we obtain
$$
\begin{aligned}
II_{\bm{p}}(\bm{e}_1,\bm{e}_1) =(h_1+\lambda)\bm\nu_1+h_2\bm\nu_2,\quad II_{\bm{p}}(\bm{e}_1,\bm{e}_2)=\mu\bm\nu_2,\quad II_{\bm{p}}(\bm{e}_2,\bm{e}_2)=(h_1-\lambda)\bm\nu_1+h_2\bm\nu_2.
\end{aligned}
$$

We now construct Monge coordinates adapted to these orthonormal tangent and normal frames. 
By an orthogonal transformation of $\R^4$, followed by the use of the orthogonal projection onto $T_{\bm{p}}M$ as local coordinates, the surface can be written in
a Monge normal form
$$
X(u,v)=\left(u,v,f(u,v),g(u,v)\right),
$$
where $f$ and $g$ have no constant and no linear terms, 
$$
X_u(0,0)=\bm{e}_1, \quad X_v(0,0)=\bm{e}_2,
$$
and the third and fourth coordinate directions are $\bm\nu_1$ and $\bm\nu_2$, respectively.

In these adapted Monge coordinates,
$$
X_{ij}(0,0) = f_{ij}(0,0)\bm\nu_1+g_{ij}(0,0)\bm\nu_2,\quad i,j\in\{u,v\}.
$$
Hence
\begin{equation*}
%\label{eq:Monge_II}
II_{\bm{p}}(\bm{e}_i,\bm{e}_j) = II_{\bm{p}}\left(X_i(0,0),X_j(0,0)\right) = f_{ij}(0,0)\bm\nu_1+g_{ij}(0,0)\bm\nu_2.
\end{equation*}
Comparison with the preceding expressions for the second fundamental form gives
\begin{equation*}
%\label{eq:2-jet}
\begin{gathered}
f_{uu}(0,0)=h_1+\lambda, \quad f_{uv}(0,0)=0, \quad f_{vv}(0,0)=h_1-\lambda, \\
g_{uu}(0,0)=h_2, \quad g_{uv}(0,0)=\mu, \quad g_{vv}(0,0)=h_2.
\end{gathered}
\end{equation*}
Therefore,
\begin{align*}
f(u,v) &= \frac{1}{2}\left\{h_1(u^2+v^2) +\lambda(u^2-v^2) \right\} +O(u,v)^3, \\
g(u,v) &= \frac{1}{2}\left\{h_2(u^2+v^2)+2\mu uv\right\}+O(u,v)^3.
\end{align*}
This proves the theorem.
\end{proof}

%%%%%%%%%%%%%%%%%%%%%%%%%%%%%%%%%

\section{Applications to the 2-jet geometry of surfaces in $\R^4$}
\label{sec:2-jet_geometry}

In this section, we describe the 2-jet geometry of $M$ in terms of the curvature-ellipse-based Monge normal form given by Theorem~\ref{thm:Euclidean_Monge_normal_form}.

The 2-jet geometry of surfaces in $\R^4$ has been well established in, for instance, \cite{BT2002, MRFR1995, NBT2007}, though their primary focus lies elsewhere.
For a systematic account of the differential geometry of surfaces in $\R^4$ from the viewpoint of singularity theory, see \cite[Chapter 7]{IFRT2015}.

Following \cite{MRFR1995}, points $\bm{p}\in M$ are classified according to the position of $\bm{0}_{\bm{p}}$ with respect to $\mathcal{E}_{\bm{p}}$ which is not degenerate, where $\bm{0}_{\bm{p}}$ denotes the origin of $N_{\bm{p}}M$.  
The point $\bm{p}$ is called \textit{elliptic}, \textit{parabolic} or \textit{hyperbolic} if $\bm{0}_{\bm{p}}$ is inside, on or outside $\mathcal{E}_{\bm{p}}$, respectively. 
When $\mathcal{E}_{\bm{p}}$ degenerates to a \textR{radial} line segment, $\bm{p}$ is an \textit{inflection} point of $M$. 
Moreover, an inflection point $\bm{p}$ is of real type, imaginary type, or flat type when $\bm{0}_{\bm{p}} \in \mathcal{E}_{\bm{p}}$, $\bm{0}_{\bm{p}}\not\in\mathcal{E}_{\bm{p}}$, or $\bm0_{\bm{p}}$ is an endpoint of $\mathcal{E}_{\bm{p}}$, respectively.

The curvature ellipse $\mathcal{E}_{\bm{p}}$ is regarded as the image of the unit circle in the tangent plane $T_{\bm{p}}M$ under a pair of quadratic forms $(Q_1,Q_2)$ associated with the second fundamental form $II_{\bm{p}}$.
Writing a tangent vector as $\bm w=u\bm{e}_1+v\bm{e}_2$,
the second fundamental form determines a pair of binary quadratic forms
$$
Q_i(u,v)=l_i u^2+2m_i u v+n_i v^2, \quad i=1,2.
$$
In \cite{BN1998}, the pencil of binary forms determined by $(Q_1,Q_2)$ is used.
A binary quadratic form $A u^2+2B u v+C v^2$ is represented by its coefficients $(A,B,C)\in\R^3$.
The equation $B^2-AC=0$ defines the cone of degenerate binary quadratic forms in $\R^3$.
After projectivization, this cone determines a conic in $\R P^2$.
If $Q_1$ and $Q_2$ are linearly independent, their pencil determines
a projective line in $\R P^2$.
This line meets the conic in 0, 1, or 2 points according as $\delta(\bm{p})$ is negative, zero, or positive, respectively, where
\begin{equation}
\label{eq:delta}
\delta(\bm{p}) = (l_1 n_2-n_1 l_2)^2-4(l_1 m_2 - m_1 l_2)(m_1 n_2 - n_1 m_2)
\end{equation}
If $\delta(\bm{p})$ is negative, zero, or positive, then a point $\bm{p} \in M$ is elliptic, parabolic, or hyperbolic, respectively. 
Besides, if the rank of $\mathcal A(\bm{p})$ is 1, so that $Q_1$ and $Q_2$ are linearly dependent, then $\bm p$ is an inflection point, where 
\begin{equation*}
%\label{eq:matrix_inf}
\mathcal{A}(\bm{p}) = 
\begin{pmatrix}
l_1 & m_1 & n_1 \\
l_2 & m_2 & n_2
\end{pmatrix}.
\end{equation*}
We remark that if $(Q_1,Q_2)$ is equivalent to $(x^2,0)$ or $(0,0)$ under the action of $GL(2,\R) \times GL(2,\R)$, then $\bm{p}$ is said to be a \textit{degenerate inflection} point (see \cite{BT2002, NBT2007}). 
We also remark that if $(Q_1,Q_2)$ is equivalent to $(x^2,0)$, then $\mathcal{E}_{\bm{p}}$ is a \textR{radial} line segment and $\bm{0}_{\bm{p}}$ is an endpoint of $\mathcal{E}_{\bm{p}}$, that is, $\bm{p}$ is an inflection point of flat type. 
On the other hand, if $(Q_1, Q_2)$ is equivalent to $(0,0)$, then $\mathcal{E}_{\bm{p}}$ degenerates to a point $\bm{0}_{\bm{p}}$, and hence $\bm{p}$ is not an inflection point \textR{in the sense of \cite{MRFR1995}. } 

Moreover, by the definition of types of $\bm{p}$, if $\bm{p}$ is an elliptic, parabolic, or hyperbolic point of $M$, then there are 0, 1, or 2 tangent lines to $\mathcal{E}_{\bm{p}}$ from $\bm{0}_{\bm{p}}$, respectively. 
%In \cite{NBT2007}, they introduced that at a hyperbolic point, there are exactly two directions in $N_{\bm{p}}M$, labelled \textit{binormal directions}. 
%The two binormal directions coincide at a parabolic point. 
%A hyperplane orthogonal to a binormal direction is called an \textit{osculating hyperplane}. 

Let $\bm{w}(\theta)\in S(T_{\bm{p}}M)$, $\theta\in[0,2\pi)$, be a regular parameterization of the unit circle in $T_{\bm{p}}M$.
%The torsion $\tau=\tau(\bm{p},\theta)$ of $\kappa_n(\bm{p},\bm{w}(\theta))$ is called the \textit{normal torsion} of $M$ in the direction $\bm{w}(\theta)$ at $\bm{p}$. 
A direction $\bm{w}(\theta_0) \in T_{\bm{p}}M$ along which 
$$
\kappa_n(\bm{p},\bm{w}(\theta_0))\quad \text{and} \quad \dfrac{\partial}{\partial \theta}\kappa_n(\bm{p},\bm{w}(\theta_0))
$$
are parallel is an \textit{asymptotic direction} of $M$ at $\bm{p}$, where $\bm{w}(\theta)$ and $-\bm{w}(\theta)$ represent the same direction (see \cite{GMRFR2000, MRFR1995}).  
Equivalently, $\bm{w}(\theta_0)$ is an asymptotic direction if the tangent line to $\mathcal{E}_{\bm{p}}$ at $\kappa_n(\bm{p},\bm{w}(\theta_0))$ passes through $\bm{0}_{\bm{p}}$ (see \cite{NBT2007}).
Therefore, if $\bm{p}$ is neither an inflection point nor a degenerate inflection point, then there are 0, 1 or 2 asymptotic directions at $\bm{p}$ according as $\bm{p}$ is an elliptic, parabolic or hyperbolic point of $M$, respectively. 
If $\bm{p}$ is an inflection point, then every direction in $T_{\bm{p}}M$ is asymptotic (see \cite{MRFR1995}). 
In \cite{BT2002}, they introduced the binary differential equation 
\begin{equation}
\label{eq:BDE_of_asymptotic}
(l_1 m_2-m_1l_2)du^2+(l_1 n_2-n_1l_2)du dv+(m_1n_2-n_1m_2)dv^2=0,
\end{equation}
and showed that the asymptotic directions of $M$ at $\bm{p}$ are solutions of \eqref{eq:BDE_of_asymptotic}. 
The discriminant of \eqref{eq:BDE_of_asymptotic} coincides with $\delta(\bm{p})$. 

For a unit normal vector $\bm\nu=\cos\phi \bm{\nu}_1+\sin\phi\bm{\nu}_2 \in S(N_{\bm{p}}M)$ ($\phi\in(-\pi/2,\pi/2]$), we define a symmetric bilinear form 
$$
II_{\bm{p},\bm\nu} \colon T_{\bm{p}}M \times T_{\bm{p}}M \to \R \quad\text{by}\quad II_{\bm{p},\bm\nu}(\bm{w}_1,\bm{w}_2)=\langle II_{\bm{p}}(\bm{w}_1,\bm{w}_2), \bm\nu \rangle.\quad (\bm{w}_i \in T_{\bm{p}}M)
$$
Hence, we obtain, in terms of the tangent frame $\{X_u,X_v\}$ of $T_{\bm{p}}M$, 
\begin{equation}
\label{eq:II}
\begin{split}
\left[II_{\bm{p},\bm\nu}\right] &= 
\begin{pmatrix}
II_{\bm{p},\bm\nu}(X_u,X_u) & II_{\bm{p},\bm\nu}(X_u,X_v) \\
II_{\bm{p},\bm\nu}(X_u,X_v) & II_{\bm{p},\bm\nu}(X_v,X_v)
\end{pmatrix} \\
&=
\begin{pmatrix}
l_1\cos\phi+l_2\sin\phi & m_1\cos\phi+m_2\sin\phi \\
m_1\cos\phi+m_2\sin\phi & n_1\cos\phi+n_2\sin\phi
\end{pmatrix}.
\end{split}
\end{equation}
A straightforward calculation shows that 
\begin{equation}
\label{eq:det_II}
\text{det}[II_{\bm{p},\bm\nu}] = (l_1 n_1-m_1^2)\cos^2\phi+(l_1n_2-2m_1m_2+n_1l_2)\cos\phi\sin\phi+(l_2n_2-m_2^2)\sin^2\phi.
\end{equation}
Moreover, regarding $\det[II_{\bm p,\bm\nu}]$ as a homogeneous quadratic form in $(\cos\phi,\sin\phi)$, we see that its discriminant is
\begin{equation}
\label{eq:disc_det_II}
\text{disc}(\text{det}[II_{\bm{p},\bm\nu}])= \delta(\bm{p}).
\end{equation}

We set the \textit{shape operator $S_{\bm{p},\bm\nu}\colon T_{\bm{p}}M\to T_{\bm{p}}M$ in the normal direction $\bm\nu$} by
$$
I_{\bm{p}}(S_{\bm{p},\bm\nu}(\bm{w}_1),\bm{w}_2)
=
II_{\bm{p},\bm\nu}(\bm{w}_1,\bm{w}_2)
$$
It is self-adjoint with respect to $I_{\bm p}$ (see \cite{Tari2009}).
This is the analogue, in the normal direction $\bm\nu$, of the usual shape operator of a surface in $\R^3$.
We denote the eigenvalues of $[S_{\bm p,\bm\nu}]$ by
$\kappa_1(\bm{p},\bm\nu)$ and $\kappa_2(\bm{p},\bm\nu)$.
Then 
\begin{equation}
\label{eq:det_S}
\det [S_{\bm{p},\bm\nu}]
=
\frac{\det[II_{\bm{p},\bm\nu}]}{\det[I_{\bm{p}}]}
=
\kappa_1(\bm{p},\bm\nu)\kappa_2(\bm{p},\bm\nu).
\end{equation}

The condition that $\text{det}[II_{\bm{p},\bm\nu}]=0$ is equivalent to the condition that there exists a non-zero vector $\bm{w}_0 \in T_{\bm{p}}M$ such that $II_{\bm{p},\bm\nu}(\bm{w}_0,\bm{w})=0$ for every $\bm{w} \in T_{\bm{p}}M$. 
Therefore, regarding $\bm{w}_0$ as the kernel direction of $[II_{\bm{p},\bm\nu}]$, we obtain $\langle II_{\bm{p}}(\bm{w}_0,\bm{w}_0),\bm\nu \rangle=0$, and thus $\langle II_{\bm{p}}(\bm{w}_0,\bm{w}_0^{\perp}),\bm\nu \rangle=0$. 
We take $\bm{w}(\theta)\in S(T_{\bm{p}}M)$, $\theta\in[0,2\pi)$, to be a regular parameterization of the unit circle in $T_{\bm{p}}M$. 
We obtain 
\begin{equation}
\label{eq:kappa_and_kappa_theta}
\kappa_n(\bm{p}, \bm{w}(\theta))=II_{\bm{p}}(\bm{w}(\theta),\bm{w}(\theta))\quad \text{and} \quad \dfrac{\partial}{\partial \theta} \kappa_n(\bm{p},\bm{w}(\theta))=2a(\theta)II_{\bm{p}}(\bm{w}(\theta),\bm{w}(\theta)^{\perp}),
\end{equation}
where 
$$
a(\theta)\ne0\quad \text{with} \quad \frac{\partial}{\partial \theta}\bm{w}(\theta) = a(\theta)\bm{w}(\theta)^{\perp}. 
$$
The equations \eqref{eq:kappa_and_kappa_theta}, for the kernel direction $\bm{w}_0=\bm{w}(\theta_0)$, imply that 
\begin{equation}
\label{eq:II_p_nu}
\begin{split}
II_{\bm{p},\bm\nu}(\bm{w}_0,\bm{w}_0) &= \langle\kappa_n(\bm{p},\bm{w}(\theta_0)),\bm\nu \rangle=0, \\
II_{\bm{p},\bm\nu}(\bm{w}_0,\bm{w}_0^{\perp}) & = \langle II_{\bm{p}}(\bm{w}_0,\bm{w}_0^{\perp}),\bm\nu\rangle = \dfrac1{2a(\theta_0)}\left\langle \dfrac{\partial}{\partial \theta}\kappa_n(\bm{p},\bm{w}(\theta_0)),\bm\nu \right\rangle =0.
\end{split}
\end{equation}
It follows from \eqref{eq:II_p_nu} that both
$$
\kappa_n(\bm{p},\bm{w}(\theta_0)) \quad\text{and}\quad \frac{\partial}{\partial\theta}\kappa_n(\bm{p},\bm{w}(\theta_0))
$$
belong to the one-dimensional subspace $\bm\nu^\perp$ of $N_{\bm{p}}M$.
Hence, they are parallel, and therefore the kernel direction $\bm{w}_0$ of $[II_{\bm{p},\bm\nu}]$ is an asymptotic direction at $\bm{p}$. 
%Moreover, $\bm\nu^\perp$ is the direction of the tangent line to $\mathcal{E}_{\bm{p}}$ at $\kappa_n(\bm{p},\bm{w}(\theta_0))$ passing through $\bm{0}_{\bm{p}}$. 
\textR{Note that at an inflection point $\bm{p}$ or at a point $\bm{p}$ where the curvature ellipse degenerates to a point other than $\bm{0}_{\bm{p}}$, there exists a unit normal vector $\bm\nu_0$ such that $[II_{\bm{p},\bm\nu_0}]=\bm{0}$, whereas at a point $\bm{p}$ where the curvature ellipse degenerates to $\bm{0}_{\bm{p}}$, $[II_{\bm{p},\bm\nu}]=\bm{0}$ holds for every unit normal vector $\bm\nu$. 
Hence, every tangent direction at $\bm{p}$ is asymptotic in all these cases. }

Conversely, let $\bm{w}_0=\bm{w}(\theta_0)$ be an asymptotic direction at $\bm{p}$. 
Then the tangent line $\ell$ to $\mathcal{E}_{\bm{p}}$ at $\kappa_n(\bm{p},\bm{w}(\theta_0))$ passes through $\bm{0}_{\bm{p}}$. 
Let $\bm\nu\in N_{\bm{p}}M$ be a unit normal vector to $\ell$. 
Since $\ell$ passes through $\bm{0}_{\bm{p}}$, we may regard $\ell$ as the one-dimensional subspace $\bm\nu^\perp$ of $N_{\bm{p}}M$.
By \eqref{eq:kappa_and_kappa_theta},
$$
II_{\bm{p}}(\bm{w}_0,\bm{w}_0) =\kappa_n(\bm{p},\bm{w}(\theta_0))\in\bm\nu^\perp \quad \text{and}\quad 
II_{\bm{p}}(\bm{w}_0,\bm{w}_0^\perp) =\frac{1}{2a(\theta_0)} \frac{\partial}{\partial\theta} \kappa_n(\bm{p},\bm{w}(\theta_0)) \in\bm\nu^\perp.
$$
Therefore, 
$$
\langle II_{\bm{p}}(\bm{w}_0,\bm{w}_0),\bm\nu\rangle = II_{\bm{p},\bm\nu}(\bm{w}_0,\bm{w}_0) = 0\quad \text{and} \quad \langle II_{\bm{p}}(\bm{w}_0,\bm{w}_0^{\perp}),\bm\nu \rangle = II_{\bm{p},\bm\nu}(\bm{w}_0,\bm{w}_0^\perp) =0.
$$
Since $\{\bm{w}_0,\bm{w}_0^\perp\}$ is an orthonormal basis of
$T_{\bm{p}}M$, it follows that $II_{\bm{p},\bm\nu}(\bm{w}_0,\bm{w})=0$ for every $\bm{w}\in T_{\bm{p}}M$.
Hence $\bm{w}_0$ is the kernel direction of $[II_{\bm{p},\bm\nu}]$.
Such a normal direction $\bm\nu$, regarded up to sign, is called a
\textit{binormal direction} of $M$ at $\bm{p}$.
At a hyperbolic point, there are exactly two binormal directions, whereas at a parabolic point two binormal directions coincide (see  \cite{MRFR1995,NBT2007}). 
\textR{Equivalently, a normal direction $\bm\nu$ is a binormal direction if and only if the height function in the direction $\bm\nu$ has a degenerate critical point at $\bm{p}$. 
This characterization also applies at inflection points  (see \cite{MRFR1995}) and at points where the curvature ellipse degenerates to a point, even though a unique asymptotic direction may not be determined at such points.}
It follows that the binormal directions of $M$ at $\bm{p}$, like the asymptotic directions, are the solutions $\text{det}[II_{\bm{p},\bm\nu}]=0$ in the normal direction $[\cos\phi:\sin\phi]$. 

After a Euclidean motion, we may assume that $\bm{p}=\bm{0}$ and write $M$ locally in the Monge normal form given by Theorem~\ref{thm:Euclidean_Monge_normal_form}:
$$
X:(\R^2,\bm{0})\to(\R^4,\bm{0}),\quad X(u,v)=(u,v,f(u,v),g(u,v)).
$$
Then we have an adapted orthonormal frame $\{\bm{e}_1,\bm{e}_2,\bm\nu_1,\bm\nu_2\}$ of $\R^4$, where $\{\bm{e}_1,\bm{e}_2\}$ is an orthonormal frame of $T_{\bm{p}}M$ and $\{\bm\nu_1,\bm\nu_2\}$ is an orthonormal frame of $N_{\bm{p}}M$. 
In particular, 
$$
X_u(0,0)=\bm{e}_1,
\qquad
X_v(0,0)=\bm{e}_2,
$$
and thus $\text{det}[I_{\bm{p}}]=E(0,0) G(0,0) - F(0,0)^2 = 1$. 
Moreover, 
\begin{align*}
II_p(\bm{e}_1,\bm{e}_1) &=(h_1+\lambda)\bm\nu_1+h_2\bm\nu_2, \\
II_p(\bm{e}_1,\bm{e}_2) &=\mu\bm\nu_2, \\
II_p(\bm{e}_2,\bm{e}_2) &=(h_1-\lambda)\bm\nu_1+h_2\bm\nu_2. 
\end{align*}

\begin{proposition}
\label{prop:second}
We have
\begin{equation}
\label{eq:second}
II_{\bm{p}}=((h_1+\lambda)du^2+(h_1-\lambda)dv^2)\bm\nu_1+(h_2du^2+2\mu du dv+h_2 dv^2)\bm\nu_2,
\end{equation}
\begin{equation}
\label{eq:matrix_of_II}
[II_{\bm{p},\bm\nu}]=
\begin{pmatrix}
(h_1+\lambda)\cos\phi+h_2\sin\phi & \mu\sin\phi \\
\mu\sin\phi & (h_1-\lambda)\cos\phi+h_2\sin\phi
\end{pmatrix},
\end{equation}
\begin{equation}
\label{eq:det_second}
\mathrm{det}[II_{\bm{p},\bm\nu}]=(h_1^2-\lambda^2)\cos^2\phi+2h_1 h_2 \cos\phi\sin\phi+(h_2^2-\mu^2)\sin^2\phi = \mathrm{det} [S_{\bm{p},\bm\nu}],
\end{equation}
and 
\begin{equation}
\label{eq:delta}
\mathrm{disc}(\mathrm{det}[II_{\bm{p},\bm\nu}]) = 4(h_1^2\mu^2+h_2^2\lambda^2-\lambda^2\mu^2) = \delta(\bm{p}).
\end{equation}
\end{proposition}
\begin{proof}
The first equation \eqref{eq:second} immediately follows from the definition of $II_{\bm{p}}$, the second equation \eqref{eq:matrix_of_II} immediately follows from \eqref{eq:II}, the third equation \eqref{eq:det_second} follows from \eqref{eq:det_II} and \eqref{eq:det_S} and the fourth equation \eqref{eq:delta} follows from \eqref{eq:disc_det_II} and \eqref{eq:det_second}.
\end{proof}

Since the type of $\bm{p}$ is determined by the sign of $\delta(\bm{p})$, we obtain the following corollary:
\begin{corollary}
Suppose that $\bm{p}$ is neither an inflection point nor a degenerate inflection point of $M$.
A point $\bm{p}$ is an elliptic, parabolic, or hyperbolic point of $M$ if and only if 
$$
h_1^2\mu^2+h_2^2\lambda^2-\lambda^2\mu^2
$$
is negative, zero, or positive, respectively.
\end{corollary}

By \eqref{eq:BDE_of_asymptotic}, the following corollary holds:
\begin{corollary}
\label{cor:BDE_of_asymptotic}
The asymptotic directions of $M$ at $\bm{p}$ are solutions of the binary differential equation
\begin{equation*}
(h_1+\lambda)\mu du^2+2h_2\lambda du dv+(-h_1+\lambda)\mu dv^2=0.
\end{equation*} 
\end{corollary}

By \eqref{eq:det_second}, the following corollary holds:
\begin{corollary}
\label{cor:BDE_of_binormal}
The binormal directions $\bm\nu=s\bm\nu_1+t\bm\nu_2$ of $M$ at $\bm{p}$ are the solutions of the homogeneous quadratic equation 
$$
(h_1^2-\lambda^2)s^2+2h_1 h_2 s t+(h_2^2-\mu^2)t^2=0.
$$
\end{corollary}

%\begin{proposition}
%\label{prop:inflection}
%A point $\bm{p}$ is an inflection point of $M$ if and only if one of the following hold:
%\begin{enumerate}
%\item 
%$\mathcal{E}_{\bm{p}}$ degenerates to a point;
%
%\item
%$\mathcal{E}_{\bm{p}}$ degenerates to a line segment and the line containing this segment passes through $\bm{0}_{\bm{p}}$.
%\end{enumerate}
%
%\end{proposition}
%\begin{proof}
%It immediately follows form the definition of $\mathcal{A}(\bm{p})$ that we have 
%\begin{equation}
%\label{eq:mathcal_A}
%\mathcal{A}(\bm{p}) = 
%\begin{pmatrix}
%h_1+\lambda & 0 & h_1-\lambda \\
%h_2 & \mu & h_2
%\end{pmatrix}.
%\end{equation}
%Hence, $\bm{p}$ is an inflection point, which is equivalent to this matrix is of rank 1 if and only if  
%$$
%(h_1+\lambda)\mu = h_2\lambda=(-h_1+\lambda)\mu=0\quad \text{and}\quad (h_1, h_2, \lambda, \mu)\ne(0,0,0,0).
%$$
%This condition is equivalent to one of the following hold:
%\begin{itemize}
%\item[(a)]
%$h_1 \ne0$, $h_2=0$, $\lambda \ne 0$, and $\mu=0$;
%\item[(b)]
%$h_1 \ne0$, $h_2=\lambda =\mu=0$; 
%\item[(c)]
%$h_1=0$, $h_2\ne0$, $\lambda=0$, and $\mu\ne0$;
%\item[(d)]
%$h_1=0$, $h_2\ne0$, $\lambda=\mu=0$.
%\end{itemize}
%The condition (b) and (d) are equivalent to (1), whereas (a) and (c) are equivalent to (2).
%\end{proof}

As mentioned above, if $\bm{p}$ is neither an inflection point nor a degenerate inflection point, then the number of asymptotic directions and binormal directions at $\bm{p}$ are determined by the type of $\bm{p}$. 
Hence, we have the following corollary:
\begin{corollary}
Suppose that $\bm{p}$ is neither an inflection point nor a degenerate inflection point of $M$.
Then the number of distinct asymptotic directions, and equivalently the number of distinct binormal directions, at $\bm{p}$ is $0$, $1$, or $2$ according as
$$
h_1^2\mu^2+h_2^2\lambda^2-\lambda^2\mu^2
$$
is negative, zero, or positive, respectively.
\end{corollary}

%%%%%%%%%%%%%%%%%%%%%%%%%%%%%%%%%

\section{Applications to projections into 3-space}
\label{sec:projection}

The study of singularities of projections of surfaces in $\R^4$ into $3$-spaces goes back to Mond \cite{Mond1982, Mond1985}. 
He used the projection-genericity/transversality framework developed by David \cite{David1983}. 
The geometry and duality of families of such projections were further studied by Bruce and Nogueira \cite{BN1998}.
Since then, related studies have been carried out (see, for example, \cite{BROS2019, NBT2007, OST2015}).

The family of projections is given by
$$
\Pi\colon M \times S^3 \to TS^3,\quad \Pi(\bm{p},\bm{v})=(\bm{v}, \bm{p}-\langle \bm{p},\bm{v}\rangle \bm{v}).
$$
For a fixed direction $\bm{v}$, the projection can be viewed locally at a point $\bm{p}\in M$ as a map-germ 
$$
\pi_{\bm{v}}\colon (\R^2,\bm{0}) \to (\R^3,\bm{0}).
$$
The generic singularities of $\pi_{\bm{v}}$ are classified under the $\mathcal{A}$-equivalence. 
Here, two map-germ $\alpha, \beta\colon(\R^2,\bm{0})\to(\R^3,\bm{0})$ are said to be \textit{$\mathcal{A}$-equivalent} if $\beta=\Psi \circ \alpha \circ \Phi^{-1}$ for some diffeomorphisms-germs $\Phi$ and $\Psi$ of the source and target, respectively. 
Those of $\mathcal{A}_e$-codimension at most 3, the dimension  of $S^3$, are listed in Table~\ref{tab:singularities_of_projections} (see \cite{Mond1985}).

\begin{center}
\begin{table}[ht!]
\caption{Generic local singularities of the projection of $M$ to 3-space \cite{Mond1985}.}
\begin{tabular}{ccc}
\hline
Name & Normal form & $\mathcal{A}_e$-codimension \\ \hline
Immersion & $(x,y,0)$ & $0$ \\
Cross cap ($S_0$) & $(x,y^2,x y)$ & $0$ \\
$S_k^\pm$ & $(x,y^2,y^3 \pm x^{k+1}y)$, $k=1,2,3$ & $k$ \\
$B_k^\pm$ & $(x,y^2,x^2y \pm y^{2k+1})$, $k=2,3$ & $k$ \\
$C_3^\pm$ & $(x,y^2,x y^3\pm x^3 y)$, & $3$ \\
$H_k$ & $(x,x y + y^{3k-1},y^3)$, $k=2,3$ & $k$\\
$P_3$ ${}^\ast$ & $(x,x y+y^3,x y^2+c y^4), c\ne 0,\frac12, 1, \frac32$ & $3$ \\ \hline
\end{tabular}

\vspace*{3pt}
${}^\ast$The codimension of $P_3$ is that of its stratum.
\label{tab:singularities_of_projections}
\end{table}
\end{center}

The projection $\pi_{\bm{v}}$ is singular at $\bm{p}$ if and only if $\bm{v}\in T_{\bm{p}}M$. 
In this case, the image of $\pi_{\bm{v}}(M)$ is a singular surface at $ \pi_{\bm{v}}(\bm{p})$.
The singularity is a cross cap when $\bm{v}$ is a non-asymptotic direction at $\bm{p}$. 
The codimension-2 singularities appear generically on loci on $M$, whereas the codimension-3 singularities appear at isolated points on these loci (see \cite{BN1998, BT2002}). 

The projection $\pi_{\bm{v}}$ is an immersion at $\bm{p}$ if and only if $\bm{v}\not\in T_{\bm{p}}M$. 
In this case, the image $\pi_{\bm{v}}$ is a regular surface at $\pi_{\bm{v}}(\bm{p})$ (see \cite{NBT2007}).

After a Euclidean motion, we may assume that $\bm{p}=\bm{0}$ and write $M$ locally in the Monge normal form given by Theorem~\ref{thm:Euclidean_Monge_normal_form}:
$$
X\colon(\R^2,\bm{0})\to(\R^4,\bm{0}), \quad X(u,v)=(u,v,f(u,v),g(u,v)).
$$
We then have an adapted orthonormal frame $\{\bm{e}_1,\bm{e}_2,\bm\nu_1,\bm\nu_2\}$ of $\R^4$, where $\{\bm{e}_1,\bm{e}_2\}$ is an orthonormal frame of $T_{\bm{p}}M$ and $\{\bm\nu_1,\bm\nu_2\}$ is an orthonormal frame of $N_{\bm{p}}M$. 
In particular,
$$
X_u(0,0)=\bm{e}_1,
\qquad
X_v(0,0)=\bm{e}_2,
$$
and thus $\text{det}[I_{\bm{p}}]=1$. 
Via $dX_{\bm{0}}$, we identify $\R^2$ with $T_{\bm{p}}M$ by
$$
(u,v)\leftrightarrow u\bm{e}_1+v\bm{e}_2.
$$

We take the unit vectors in $S(T_{\bm{p}}M)$ as
$$
\bm{w}=\cos\theta\bm{e}_1+\sin\theta\bm{e}_2,\quad
\bm{w}^{\perp}=-\sin\theta\bm{e}_1+\cos\theta\bm{e}_2,\quad \theta\in\left(-\frac{\pi}2,\frac{\pi}2\right] 
$$
and the unit vectors in $S(N_{\bm{p}}M)$ as
$$
\bm\nu=\cos\phi\bm\nu_1+\sin\phi\bm\nu_2,\quad
\bm\nu^{\perp}=-\sin\phi\bm\nu_1+\cos\phi\bm\nu_2,\quad \phi\in\left(-\frac{\pi}2,\frac{\pi}2\right].
$$
Then the projection direction $\bm{v}$ of $\pi_{\bm{v}}$ can be written as 
$$
\bm{v}=\cos\psi \bm{w}+\sin\psi\bm\nu^{\perp},\quad \psi\in\left(-\frac{\pi}2,\frac{\pi}2\right].
$$
Setting
$$
\bm\eta_1=-\bm{w}^{\perp},\quad
\bm\eta_2=\sin\psi \bm{w}-\cos\psi \bm\nu^{\perp},\quad
\bm\eta_3=\bm\nu,
$$
we obtain the positively oriented orthonormal frame $\{\bm\eta_1,\bm\eta_2, \bm{v}, \bm\eta_3\}$ of $\R^4$ and the positively oriented orthonormal frame $\{\bm\eta_1,\bm\eta_2,\bm\eta_3\}$ of the hyperplane $\bm{v}^{\perp}$. 
Hence, the projection $\pi_{\bm{v}}$ on $M$ along $\bm{v}$ to $\bm{v}^{\perp}$ is expressed by
$$
\pi_{\bm{v}}(u,v) = (\langle X(u,v),\bm\eta_1 \rangle, \langle X(u,v),\bm\eta_2 \rangle, \langle X(u,v),\bm\eta_3 \rangle).
$$

We denote the projected surface $\pi_{\bm{v}}(M)$ by $\widetilde{M}$.
We identify $\bm{p}\in M$ with $\pi_{\bm{v}}(0,0)\in\widetilde{M}$ and regard $\bm{p}$ as the origin of $\bm{v}^\perp$. 
In particular, $\bm{p}=\bm{0}\in \R^4$.

%%%%%%%%%%%%%%%%%%%%%%%%%%%%%%%

\subsection{The geometry of projected regular surfaces}
\label{sec:projected_regular_surface}

We first consider the case $\bm{v}\not\in T_{\bm{p}}M$, which is equivalent to $\sin\psi\ne0$. 
We may assume that $\sin\psi>0$ by \textR{suitably choosing the orientations of $\bm\nu^\perp$ and $\bm\eta_2$,} if necessary. 
In this case, the projected surface $\widetilde{M}$ is a regular surface at $\bm{0}=\pi_{\bm{v}}(0,0)$.

Via $d(\pi_{\bm{v}})_{\bm{p}}$, we identify $T_{\bm{p}}\widetilde{M}$ with $T_{\bm{p}}M$. 
The first fundamental form of $\widetilde{M}$ is given by
$$
\widetilde{I}_{\bm{p}}\colon T_{\bm{p}}\widetilde{M} \times T_{\bm{p}}\widetilde{M} \to \R,\quad 
\widetilde{I}_{\bm{p}}(\bm{w},\bm{w})=\widetilde{E}du^2+2\widetilde{F}du dv+\widetilde{G}dv^2, \quad \bm{w}\in T_{\bm{p}}\widetilde{M},
$$
where 
$$
\widetilde{E}=\langle (\pi_{\bm{v}})_u, (\pi_{\bm{v}})_u\rangle, \quad 
\widetilde{F}=\langle (\pi_{\bm{v}})_u, (\pi_{\bm{v}})_v\rangle, \quad
\widetilde{G}=\langle (\pi_{\bm{v}})_v, (\pi_{\bm{v}})_v\rangle. 
$$
By direct calculation, we obtain 
\begin{equation*}
\begin{split}
\widetilde{E}(0,0)&=\sin^2\theta+\cos^2\theta\sin^2\psi,\\ 
\widetilde{F}(0,0)&=-\cos\theta\sin\theta\cos^2\psi,\\
\widetilde{G}(0,0)&=\cos^2\theta+\sin^2\theta\sin^2\psi.
\end{split}
\end{equation*}
It follows that $\text{det}\left[\widetilde{I}_{\bm{p}}\right] = \sin^2\psi \ne0$ at $\bm{0}$. 
Thus, we directly confirm that $\widetilde{M}$ is regular at $\bm{0}$.

The unit normal vector field to $\widetilde{M}$ is given by
$$
\widetilde{\bm{n}}=\dfrac{(\pi_{\bm{v}})_u \times (\pi_{\bm{v}})_v}{\lVert (\pi_{\bm{v}})_u \times (\pi_{\bm{v}})_v \rVert},
$$
and we obtain
$$
\widetilde{\bm{n}}(0,0) = (0,0,1).
$$
Hence, 
$$
\widetilde{\bm{n}}(0,0)=\bm\nu
$$
under the identification of $\bm{v}^{\perp}$ with the subspace of $\R^4$ spanned by $\{\bm\eta_1,\bm\eta_2,\bm\eta_3\}$.

The second fundamental form of $\widetilde{M}$ is given by
$$
\widetilde{II}_{\bm{p}}\colon T_{\bm{p}}\widetilde{M} \times T_{\bm{p}}\widetilde{M} \to \R,\quad 
\widetilde{II}_{\bm{p}}(\bm{w},\bm{w}) = \widetilde{l} du^2+2\widetilde{m} du dv+\widetilde{n} dv^2,\quad \bm{w}\in T_{\bm{p}}\widetilde{M},
$$
where 
$$
\widetilde{l}=\langle (\pi_{\bm{v}})_{uu},\widetilde{\bm{n}} \rangle,\quad
\widetilde{m}=\langle (\pi_{\bm{v}})_{uv},\widetilde{\bm{n}} \rangle,\quad 
\widetilde{n}=\langle (\pi_{\bm{v}})_{vv},\widetilde{\bm{n}} \rangle.
$$
A direct calculation gives that
\begin{equation}
\label{eq:coefficients_of_II}
\begin{split}
\widetilde{l}(0,0)&=(h_1 +\lambda)\cos\phi+h_2\sin\phi,\\
\widetilde{m}(0,0)&=\mu\sin\phi,\\
\widetilde{n}(0,0)&=(h_1-\lambda)\cos\phi+h_2\sin\phi.
\end{split}
\end{equation}
It follows from \eqref{eq:det_S}, \eqref{eq:det_second} and \eqref{eq:coefficients_of_II} that, at $\bm{0}$,  
$$
\text{det}\left[\widetilde{II}_{\bm{p}}\right] = (h_1^2-\lambda^2)\cos^2\phi+2h_1 h_2 \cos\phi\sin\phi+(h_2^2-\mu^2)\sin^2\phi = \text{det}[S_{\bm{p},\bm\nu}]. 
$$
Therefore, the Gaussian curvature $\widetilde{K}$ of $\widetilde{M}$ at $\bm{0}$ is given by
$$
\widetilde{K}=\dfrac{\text{det}\left[\widetilde{II}_{\bm{p}}\right]}{\text{det}\left[\widetilde{I}_{\bm{p}}\right]} = \dfrac{\text{det}[S_{\bm{p},\bm\nu}]}{\sin^2\psi}.
$$
Hence, we have the following proposition:
\begin{proposition}
\label{prop:parabolicity_of_widetilde_M}
A point $\bm{0} \in \widetilde{M}$ is an elliptic, parabolic, or hyperbolic point if and only if $\mathrm{det}\left[S_{\bm{p},\bm\nu}\right]$ is positive, zero,  or negative, respectively. 
\end{proposition}

A direct calculation shows that the mean curvature $\widetilde{H}$ of $\widetilde{M}$ is given by
\begin{equation*}
%\label{eq:mean}
%\widetilde{H} = -\dfrac{B(\theta,\phi,\psi)}{2\sin^2\psi},
\textR{\widetilde{H} = -\dfrac{B(\theta,\phi,\psi)}{2A(\psi)}}
\end{equation*}
where 
$$
\textR{A(\psi)=\sin^2\psi=\text{det}\left[\widetilde{I}_{\bm{p}}\right],} \quad 
B(\theta,\phi,\psi)=\cos^2\psi\langle II_{\bm{p}}(\bm{w},\bm{w}),\bm\nu\rangle +2\sin^2\psi\langle \bm{H},\bm\nu\rangle.
$$
It follows that we have the following proposition:
\begin{proposition}
\textR{At a point $\bm{0} \in \widetilde{M}$}, $\widetilde{H}=0$ if and only if $B(\theta,\phi,\psi)=0$.
\end{proposition}

Let $X\colon(\R^2,\bm{q})\to(\R^3,\bm{p})$ be a local parameterization of a regular surface $S\subset \R^3$.  
The distance-squared function on $S$ from a point $\bm{a}\in\R^3$ is given by 
$$
d_{\bm{a}}\colon (\R^2,\bm{q})\to \R, \quad d_{\bm{a}}(u,v) = \lVert X(u,v) - \bm{a} \rVert^2.
$$
The focal geometry of $S$ can be characterized in terms of the singularities of $d_{\bm{a}}$ as follows (see \cite{Porteous1971}; see also \cite{FH2012-1} for another example): 
\begin{align*}
A_{\ge1}\text{-singularity at }\bm{q}& \Leftrightarrow \bm{a}= \bm{p} + \R \bm{n}(\bm{q}), \\
A_{\ge2}\text{-singularity at }\bm{q}& \Leftrightarrow \textR{\bm{p}\text{ is a  non-umbilic point and }\bm{a}=\bm{p}+\dfrac1{\kappa_i(\bm{q})} \bm{n}(\bm{q})\quad (\kappa_i(\bm{q})\ne0)}, \\
A_{\ge3}\text{-singularity at }\bm{q}& \Leftrightarrow \textR{\bm{p}\text{ is a  non-umbilic point, }\bm{a}=\bm{p}+\dfrac1{\kappa_i(\bm{q})} \bm{n}(\bm{q})\quad (\kappa_i(\bm{q})\ne0)}\\
&\textR{\quad\text{ and }\bm{p}\text{ is a ridge point relative to }\bm{v}_i}, \\
D_{\ge4}\text{-singularity at }\bm{q}& \Leftrightarrow \bm{p}\text{ is a  non-flat umbilic point and }\bm{a}=\bm{p}+\dfrac1{\kappa(\bm{q})} \bm{n}(\bm{q}),
\end{align*}
where $\bm{n}$ is the unit normal vector field to $S$, $\kappa_i$ $(i=1,2)$ are the principal curvatures of $S$, $\bm{v}_i$ are the principal directions corresponding to $\kappa_i$ and $\kappa$ is the common principal curvature at an umbilic point. 
The notation for $A_k$- and $D_k$-singularities follows that of \cite{Arnold1986}. 
We now apply these facts to the projected surface $\widetilde{M}$.
Recall that $\bm{p}=\bm{0}$ and that $\widetilde{\bm{n}}(0,0)=\bm\nu$.

The distance-squared function on $\widetilde{M}$ from a point $\bm{a}\in \bm{v}^{\perp}\subset \R^4$ is given by
$$
d_{\bm{a}}\colon (\R^2,\bm{0})\to\R,\quad d_{\bm{a}}(u,v) = \lVert \pi_{\bm{v}}(u,v)-\bm{a}\rVert^2.
$$

We easily show that $d_{\bm{a}}$ has an $A_{\ge1}$-singularity at $(0,0)$ if and only if
$$
\langle (\pi_{\bm{v}})_u(0,0), \bm{p}-\bm{a}\rangle = \langle (\pi_{\bm{v}})_v(0,0), \bm{p}-\bm{a}\rangle = 0
$$
which is equivalent to 
\begin{equation}
\label{eq:A_1-singularity_of_d}
\begin{array}{ccccc}
 & \bm{v}^\perp & \subset & \R^4 & \\ [-1mm]
 & \rotatebox[origin=c]{90}{$\in$}
 & & \rotatebox[origin=c]{90}{$\in$} & \\ [-1mm]
\bm{a}= & r\widetilde{\bm{n}}(0,0)
& = & r\bm\nu,
\end{array}
\end{equation}
where $|r|$ is the radius of the sphere $\{\bm{x}\in\bm{v}^{\perp}\colon\lVert \bm{x}-\bm{a}\rVert^2=r^2\}$. 

\textR{Assume that \eqref{eq:A_1-singularity_of_d} holds.}
Then we have 
\begin{equation}
\label{eq:Hessian_of_d}
\begin{split}
\mathcal{H}(d_{\bm{a}})(0,0) = 
2\begin{pmatrix}
c_{20} & c_{11} \\
c_{11} & c_{02} 
\end{pmatrix}
=2\left(
\begin{pmatrix}
\widetilde{E} & \widetilde{F} \\
\widetilde{F} & \widetilde{G}
\end{pmatrix}
-r
\begin{pmatrix}
\widetilde{l} & \widetilde{m} \\
\widetilde{m} & \widetilde{n}
\end{pmatrix}
\right)
=2\left(\left[\widetilde{I}_{\bm{p}}\right]-r\left[\widetilde{II}_{\bm{p}}\right]\right),
\end{split}
\end{equation}
where $\mathcal{H}(d_{\bm{a}})$ is the Hessian matrix of $d_{\bm{a}}$ and 
\begin{equation*}
%\label{eq:coefficients_Hessian_of_d}
\begin{split}
c_{20} &= \sin^2\theta+\cos^2\theta\sin^2\psi-r((h_1 +\lambda)\cos\phi+h_2\sin\phi) =\widetilde{E}-r\widetilde{l}\\
c_{11} & = -\cos\theta\sin\theta\cos^2\psi-r\mu\sin\phi = \widetilde{F}-r\widetilde{m} \\
c_{02} & = \cos^2\theta+\sin^2\theta\sin^2\psi-r((h_1-\lambda)\cos\phi+h_2\sin\phi) = \widetilde{G}-r\widetilde{n}.
\end{split}
\end{equation*}
Therefore, 
\begin{equation}
\label{eq:equation_of_kappa}
%\text{det}\mathcal{H}(d_{\bm{a}})(0,0)=4\,\text{det}\left[\widetilde{I}_{\bm{p}}\right]\left(1-r\widetilde{\kappa}_1(0,0)\right)\left(1-r\widetilde{\kappa}_2(0,0)\right),
\textR{\text{det}\mathcal{H}(d_{\bm{a}})(0,0)=4A(\psi)\left(1-r\widetilde{\kappa}_1(0,0)\right)\left(1-r\widetilde{\kappa}_2(0,0)\right),}
\end{equation}
where $\widetilde{\kappa}_i$ $(i=1,2)$ are the principal curvatures of $\widetilde{M}$. 
Moreover, 
\begin{equation}
\label{eq:det_of_Hessian_of_d}
\text{det}\mathcal{H}(d_{\bm{a}})(0,0)=4(A(\psi)+r B(\theta,\phi,\psi)+r^2Q(\phi)),
\end{equation}
where
$$
%A(\psi)=\sin^2\psi=\text{det}\left[\widetilde{I}_{\bm{p}}\right] \quad \text{and} 
\quad Q(\phi) = \text{det}\left[\widetilde{II}_{\bm{p}}\right]=\text{det}[S_{\bm{p},\bm\nu}].
$$
\textR{If $\bm{p}$ is a parabolic point of $\widetilde{M}$ then $\widetilde{\kappa}_i(0,0)\ne0$ and $\widetilde{\kappa}_j(0,0)=0$ for some $i\ne j$.
In this case, $Q(\phi)=0$, and hence 
$$
\text{det}\mathcal{H}(d_{\bm{a}})(0,0) = 4(A(\psi)+r B(\theta,\phi,\psi)).
$$
However, since $\widetilde{\kappa}_i(0,0)\ne0$, it follows from \eqref{eq:equation_of_kappa} and \eqref{eq:det_of_Hessian_of_d} that $B(\theta,\phi,\psi)\ne0$. 
Hence, 
$$
r= -\dfrac{A(\psi)}{B(\theta,\phi,\psi)} = \dfrac{1}{\widetilde{\kappa}_i(0,0)}
$$
is a simple root of \eqref{eq:det_of_Hessian_of_d}.
On the other hand, }it follows from \eqref{eq:A_1-singularity_of_d}, \eqref{eq:equation_of_kappa} and \eqref{eq:det_of_Hessian_of_d} that  
\begin{equation}
\label{eq:umbilic}
B(\theta,\phi,\psi)^2-4A(\psi)Q(\phi)=0
\end{equation} 
if and only if $r=1/\tilde{\kappa}(0,0)$ is a double root of \eqref{eq:det_of_Hessian_of_d}.
In this case, 
$$
\bm{a}=\dfrac{1}{\widetilde{\kappa}(0,0)}\widetilde{\bm{n}}(0,0)= \dfrac{1}{\widetilde{\kappa}(0,0)} \bm\nu,
$$
where $\tilde{\kappa}$ is the common principal curvature at $\bm{p}$.
Hence, the following propositions hold:

\begin{proposition}
\label{prop:A_2-singularity_of_d}
%Suppose that $\bm{0}$ is a non-parabolic point of $\widetilde{M}$. 
%Then the following conditions are equivalent: 
\textR{The following conditions are equivalent: }
\begin{enumerate}
\item 
$d_{\bm{a}}$ has an $A_{\ge2}$-singularity at $(0,0)$;
\item
\textR{$\bm{0}$} is a non-umbilic point of $\widetilde{M}$ and, for some $i=1,2$,  
\begin{equation}
\label{eq:A_2-singularity_of_d}
\begin{array}{cccccc}
 & \bm{v}^\perp & \subset & \R^4 & \\ [-1mm]
 & \rotatebox[origin=c]{90}{$\in$}
 & & \rotatebox[origin=c]{90}{$\in$} & \\ [-1mm]
\bm{a}= & \dfrac1{\widetilde{\kappa}_i(0,0)}\widetilde{\bm{n}}(0,0)
& = &\dfrac1{\widetilde{\kappa}_i(0,0)} \bm\nu & \textR{(\widetilde{\kappa}_i(0,0)\ne0)};
\end{array}
\end{equation}
\item
\textR{\eqref{eq:A_1-singularity_of_d} holds and} $r=1/\widetilde{\kappa_i}(0,0)$ is a simple root of \eqref{eq:det_of_Hessian_of_d}.
\end{enumerate}
\end{proposition}

\begin{proposition}
\label{prop:D_2-singularity_of_d}
%Suppose that $\bm{p}$ is a non-parabolic point of $\widetilde{M}$. 
%Then the following conditions are equivalent: 
\textR{The following conditions are equivalent: }
\begin{enumerate}
\item 
$d_{\bm{a}}$ has a $D_{\ge4}$-singularity at $(0,0)$;
\item
$\bm{a}=1/\widetilde{\kappa}(0,0)\widetilde{\bm{n}}(0,0)= 1/\widetilde{\kappa}(0,0) \bm\nu$ and \textR{$\bm{0}$} is a non-flat umbilic point of $\widetilde{M}$;
\item
\textR{\eqref{eq:A_1-singularity_of_d} holds and} $B(\theta,\phi,\psi)^2-4A(\psi)Q(\phi)=0$ holds.
\end{enumerate}
\end{proposition}

By \eqref{eq:Hessian_of_d}, if $c_{20}=\widetilde{E}-r\widetilde{l}\ne0$, then 
$$
\bm\xi_0=\left(c_{11}, -c_{20}\right) = c_{11}\bm{e}_1-c_{20}\bm{e}_2
$$
is a non-zero kernel vector of $\mathcal{H}(d_{\bm{a}})(0,0)$. 
If $c_{20}=\widetilde{E}-r\widetilde{l}=0$, then
$$
\bm\xi_0=\left(c_{02},-c_{11}\right) = c_{02}\bm{e}_1-c_{11}\bm{e}_2
$$
is a non-zero kernel vector of $\mathcal{H}(d_{\bm{a}})(0,0)$.
Therefore, at \textR{a non-umbilic point} $\bm{p}$, the function $d_{\bm{a}}$ has an $A_{\ge3}$-singularity at $(0,0)$ if and only if \eqref{eq:A_2-singularity_of_d} holds and $d_{\bm{a},3}(\bm\xi_0)=0$ (see~\cite{FH2012-1}), where $d_{\bm{a},3}$ is the cubic part of $d_{\bm{a}}$ given by
\begin{equation*}
%\label{eq:cubic_part_of_d}
d_{\bm{a},3}(u,v)=-\langle II_{\bm{p}}(\bm\xi,\bm\xi),\bm{v}_N\rangle \langle \bm\xi,\bm{v}_T\rangle + 2(f_3(u,v)\cos\phi+g_3(u,v)\sin\phi). 
\end{equation*}
Here, $\bm\xi=(u,v)\in\R^2\simeq T_{\bm{p}}M$, $\bm{v}_T \in T_{\bm{p}}M$ is the tangent component of $\bm{v}$ and $\bm{v}_N\in N_{\bm{p}}M$ is the normal component of $\bm{v}$. 
It follows that the following proposition holds:
\begin{proposition}
\label{prop:A_3-singularity_of_d}
%Suppose that $\bm{p}$ is a non-parabolic and non-umbilic point of $\widetilde{M}$ and \eqref{eq:A_2-singularity_of_d} holds. 
%Then the following conditions are equivalent: 
\textR{The following conditions are equivalent:}
\begin{enumerate}
\item
$d_{\bm{a}}$ has an $A_{\ge3}$-singularity at $(0,0)$; 
\item 
\textR{$\bm{0}$ is a non-umbilic point of $\widetilde{M}$, for some $i=1,2$}, $\bm{a}= 1/\widetilde{\kappa}_i(0,0) \widetilde{\bm{n}}(0,0) = 1/\widetilde{\kappa}_i(0,0) \bm\nu$ \textR{$(\widetilde{\kappa}_i(0,0)\ne0)$} and \textR{$\bm{0}$} is a ridge point relative to $\widetilde{\bm{v}}_i$, where $\widetilde{\bm{v}}_i$ is the principal direction with respect to $\widetilde{\kappa}_i$;
\item
\textR{\eqref{eq:A_1-singularity_of_d} holds, $r=1/\widetilde{\kappa_i}(0,0)$ $(\widetilde{\kappa}_i(0,0)\ne0)$ is a simple root of \eqref{eq:det_of_Hessian_of_d} and} 
$$
-\langle II_{\bm{p}}(\bm\xi_0,\bm\xi_0),\bm{v}_N\rangle \langle \bm\xi_0,\bm{v}_T\rangle + 2(f_3(\bm\xi_0)\cos\phi+g_3(\bm\xi_0)\sin\phi)=0
$$
holds, where $\bm\xi_0$ is a non-zero kernel vector of $\mathcal{H}(d_{\bm{a}})(0,0)$ \textR{for $r=1/\widetilde{\kappa}_i(0,0)$.}
\end{enumerate}
\end{proposition}

%%%%%%%%%%%%%%%%%%%%%%%%%%%%%%%%%%%%

\subsection{The geometry of projected singular surfaces}
\label{sec:projected_singular_surface}

Next, we consider the case $\bm{v}\in T_{\bm{p}}M$, which is equivalent to $\psi=0$. 
Note that 
$$
\bm{v}=\bm{w},\quad \bm\eta_1=-\bm{w}^{\perp}, \quad \bm\eta_2=-\bm\nu^{\perp},\quad \bm\eta_3=\bm\nu.
$$
In this case, the projected surface $\widetilde{M}$ is a singular surface at $\bm{0}=\pi_{\bm{v}}(0,0)$. 
Indeed, the Jacobian matrix of $\pi_{\bm{v}}$ at $(0,0)$ is given by
$$
\begin{pmatrix}
\sin\theta & -\cos\theta \\
0 & 0 \\
0 & 0
\end{pmatrix}.
$$

Setting 
$$
x=u\sin\theta-v\cos\theta,\quad y=u\cos\theta+v\sin\theta,
$$
we obtain that 
$$
\pi_{\bm{v}}(x,y)=(x, \widetilde{f}(x,y), \widetilde{g}(x,y))
$$
which satisfies 
$$
d(\pi_{\bm{v}})_{\bm{0}}(\partial_x) \ne0, \quad d(\pi_{\bm{v}})_{\bm{0}}(\partial_y)=0,
$$
where
\begin{equation}
\label{eq:2-jet_of_projected_singular_surface}
\begin{split}
\widetilde{f} &= \dfrac12\{((h_2-\mu\sin2\theta)\cos\phi+(-h_1+\lambda\cos2\theta)\sin\phi)x^2\\
& \quad -2(\mu\cos2\theta\cos\phi+\lambda\sin2\theta\sin\phi)xy \\
&\quad +((h_2+\mu\sin2\theta)\cos\phi-(h_1+\lambda\cos2\theta)\sin\phi)y^2\}+O(x,y)^3, \\
\widetilde{g} & = \dfrac12\{((h_1-\lambda\cos2\theta)\cos\phi+(h_2-\mu\sin2\theta)\sin\phi) x^2\\
& \quad +2(\lambda\sin2\theta\cos\phi-\mu\cos2\theta\sin\phi)xy\\
& \quad +((h_1+\lambda\cos2\theta)\cos\phi+(h_2+\mu\sin2\theta)\sin\phi)\}y^2 + O(x,y)^3.
\end{split}
\end{equation}

The group $\mathcal{A}$ of pairs of diffeomorphism-germs of the source and target acts on the space of map-germs by $\mathcal{A}$-equivalence.
The space of 2-jets of map-germs $(\R^2,\bm{0})\to(\R^3,\bm{0})$ is denoted by $J^2(2,3)$ and the subset of 2-jets in $J^2(2,3)$ of corank 1 is denoted by $\Sigma^1J^2(2,3)$. 
The group $\mathcal{A}$ acts on $J^2(2,3)$ in the obvious way, and this action factors via the projection $\mathcal{A} \to \mathcal{A}^k$, where $\mathcal{A}^k$ denotes the group of $k$-jets of elements of $\mathcal{A}$. 
The group $\mathcal{A}^k$ acts smoothly on $J^k(2,3)$.
In \cite{Mond1985}, Mond classified all corank-one map-germs $(\R^2,\bm{0})\to(\R^3,\bm{0})$ according to their 2-jet under the action of $\mathcal{A}^2$.
\begin{proposition}[\cite{Mond1985}]
There are four $\mathcal{A}^2$-orbits in $\Sigma^1J^2(2,3):$
$$
(x,y^2,xy),\quad (x,y^2,0),\quad (x,xy,0),\quad \mathrm{and}\quad (x,0,0).
$$
\end{proposition}

We now characterize the four $\mathcal{A}^2$-orbits of the $2$-jet of
$\pi_{\bm{v}}$ in terms of the geometry of $M$ (cf. \cite{BROS2019}).
\begin{proposition}
The 2-jet $j^2\pi_{\bm{v}}(\bm{0})$ of the local parameterization $\pi_{\bm{v}}$ of the projected singular surface $\widetilde{M}$ is $\mathcal{A}$-equivalent to: 
\begin{enumerate}
\item
$(x,y^2,xy)$ if and only if the projection direction $\bm{v}$ is not parallel to the asymptotic direction of $M$ at $\bm{p}$;
\item
$(x,y^2,0)$ if and only if $\bm{v}$ is parallel to the asymptotic direction $\bm{w}(\theta_0)$ of $M$ at $\bm{p}$ and $\bm{0}_{\bm{p}}\ne \mathcal{E}_{\bm{p}}(\theta_0)$; 
\item
$(x,xy,0)$ if and only if $\bm{v}$ is parallel to the asymptotic direction $\bm{w}(\theta_0)$ at $\bm{p}$, $\bm{0}_{\bm{p}}=\mathcal{E}_{\bm{p}}(\theta_0)$ and $\mathcal{E}_{\bm{p}}$ is a non-degenerate ellipse;
\item
$(x,0,0)$ if and only if $\bm{v}$ is parallel to the asymptotic direction $\bm{w}(\theta_0)$ at $\bm{p}$, $\bm{0}_{\bm{p}}=\mathcal{E}_{\bm{p}}(\theta_0)$ and $\mathcal{E}_{\bm{p}}$ degenerates to a line segment or a point.
\end{enumerate}
\end{proposition}
\begin{proof}
If the 2-jet of a map-germ $(\R^2,\bm{0})\to(\R^3,\bm{0})$ is given by
$$
\left(x,\frac12(a_{20}x^2+2a_{11}xy+a_{02}y^2),\frac12(b_{20}x^2+2b_{11}xy+b_{02}y^2)\right),
$$
then, by direct calculation, we obtain the following conditions.
\begin{center}
\begin{tabular}{lcl}
(a) the 2-jet is $\mathcal{A}$-equivalent to $(x,y^2,xy)$ & $\Leftrightarrow$ & $a_{11}b_{02}-a_{02}b_{11}\ne0$; \\
(b) the 2-jet is $\mathcal{A}$-equivalent to $(x,y^2,0)$ & $\Leftrightarrow$ & $a_{11}b_{02}-a_{02}b_{11}=0$, and $(a_{02},b_{02})\ne(0,0)$; \\
(c) the 2-jet is $\mathcal{A}$-equivalent to $(x,xy,0)$ & $\Leftrightarrow$ & $a_{11}b_{02}-a_{02}b_{11}=a_{02}=b_{02}=0$\\
&& and $(a_{11},b_{11})\ne(0,0)$; \\
(d) the 2-jet is $\mathcal{A}$-equivalent to $(x,0,0)$ & $\Leftrightarrow$ & $a_{11}b_{02}-a_{02}b_{11}=a_{02}=b_{02}=a_{11}=b_{11}=0$.
\end{tabular}
\end{center}

By Corollary \ref{cor:BDE_of_asymptotic} and \eqref{eq:2-jet_of_projected_singular_surface}, $a_{11}b_{02}-a_{02}b_{11}\ne0$ holds if and only if the projection direction $\bm{v}$ is not parallel to the asymptotic direction at $\bm{p}$. 
Therefore, (1) follows from (a). 

Assume that $a_{11}b_{02}-a_{02}b_{11}=0$. 
Then we have an asymptotic direction $\bm{w}=\bm{w}(\theta_0)$, and thus by \eqref{eq:curvature_ellipse_normal_form}, we have 
$$
II_{\bm{p}}(\bm{w}(\theta_0), \bm{w}(\theta_0))=\mathcal{E}_{\bm{p}}(\theta_0)=(h_1+\lambda \cos2\theta_0)\bm\nu_1+(h_2+\mu\sin2\theta_0)\bm\nu_2 \in \mathcal{E}_{\bm{p}}.
$$
It follows from \eqref{eq:2-jet_of_projected_singular_surface} that
$$
2a_{02}=\langle II_{\bm{p}}(\bm{w}(\theta_0),\bm{w}(\theta_0)),\bm\nu^\perp\rangle,
\qquad
2b_{02}=\langle II_{\bm{p}}(\bm{w}(\theta_0),\bm{w}(\theta_0)),\bm\nu\rangle.
$$
Since $\{\bm\nu,\bm\nu^{\perp}\}$ is an orthonormal basis of $N_{\bm{p}}M$, we have $a_{02}= b_{02}=0$ if and only if $II_{\bm{p}}(\bm{w}(\theta_0),\bm{w}(\theta_0))=\bm{0}$, which is equivalent to $\bm{0}_{\bm{p}}=\mathcal{E}_{\bm{p}}(\theta_0)$. 
Hence, (2) follows from (b). 

Suppose that $a_{11}b_{02}-a_{02}b_{11}=a_{02}=b_{02}=0$. 
Then we also have an asymptotic direction $\bm{w}=\bm{w}(\theta_0)$ and $\bm{0}_{\bm{p}}=\mathcal{E}_{\bm{p}}(\theta_0)$. 
It follows from \eqref{eq:2-jet_of_projected_singular_surface} that $a_{11}=b_{11}=0$ if and only if 
\begin{equation*}
\begin{pmatrix}
-\cos\phi & -\sin\phi \\
-\sin\phi & \cos\phi
\end{pmatrix}
\begin{pmatrix}
\cos2\theta_0 & 0 \\
0 & \sin2\theta_0
\end{pmatrix}
\begin{pmatrix}
\mu \\
\lambda
\end{pmatrix}
=\begin{pmatrix}
0\\
0
\end{pmatrix}.
\end{equation*}
Since the determinant of the first matrix of the left hand side of the above equation is equal to $-1$, we then obtain
\begin{equation*}
\mu\cos2\theta_0 = 0,\quad \lambda\sin2\theta_0=0.
\end{equation*}
Combining these equations with $a_{02}=b_{02}=0$, we obtain either $\lambda=0$, $\mu\ne0$, with $\mathcal{E}_{\bm{p}}$ a line segment lying on the $\bm\nu_2$-axis; or $\lambda\ne0$, $\mu=0$, with $\mathcal {E}_{\bm{p}}$ a line segment lying on the $\bm\nu_1$-axis; or \(\lambda=\mu=0\), with \(\mathcal E_{\bm p}\) a point. 
Hence, $\mathcal{E}_{\bm{p}}$ is degenerate if and only if $a_{11}=b_{11}=0$.
Therefore, (3) and (4) follow from (c) and (d), respectively. 
\end{proof}
%%%%%%%%%%%%%%%%%%%%%%%%%%%%%%%%%%%%%%

\section*{Acknowledgments}
The author would like to thank Y.~Kabata for providing the inspiration for Theorem~\ref{thm:Euclidean_Monge_normal_form}, and T.~Fukui, R.~Oset Sinha and F.~Tari for their comments. 

%%%%%%%%%%%%%%%%%%%%%%%%%%%%%%%%%%%


\begin{thebibliography}{99}

\bibitem{Arnold1986}
V.~I.~Arnol'd, S.~M.~Guse{\u\i}n-Zade and A.~N.~Varchenko,
Singularities of differentiable maps.\ Vol.~I.
\href{https://doi.org/10.1007/978-0-8176-8340-5}
{Monogr.\ Math. 82, (Birkh$\ddot{\text{a}}$user Boston, Inc., 1986.)}

\bibitem{Asperti1980}
A.~C. Asperti,
Immersions of surfaces into 4-dimensional spaces with nonzero normal curvature, 
\href{https://doi.org/10.1007/BF01789417}
{Ann. Mat. Pura Appl. \textbf{125} (1980), 313--328.}

\bibitem{Asperti1984}
A.~C. Asperti,
Minimal surfaces with constant normal curvature, 
\href{https://doi.org/10.2969/jmsj/03630375}
{J. Math. Soc. Jpn. \textbf{36} (1984), no.~3, 375--386.}

\bibitem{AFR1982}
A.~C. Asperti, D.~Ferus and L.~Rodriguez, 
Surfaces with non-zero normal curvature tensor, 
Atti Accad. Naz. Lincei, Rend. Cl. Sci. Fis. Mat. Nat. \textbf{73} (1982), 109--115.

\bibitem{BSB2012}
P. Bayard and F. S\'anchez-Bringas,
Geometric invariants of surfaces in $\R^4$, 
\href{https://doi.org/10.1016/j.topol.2011.09.014}
{Topology Appl. \textbf{159} (2012), 405--413.}

\bibitem{BROS2019}
P.~Benedini Riul and R.~Oset Sinha, 
A relation between the curvature ellipse and the curvature parabola, 
\href{https://doi.org/10.1515/advgeom-2019-0002}
{Adv. Geom. \textbf{19} (2019) 389--399.}

\bibitem{BN1998}
J.~W.~Bruce and A.~C.~Nogueira,
Surfaces in $\R^4$ and duality, 
\href{https://doi.org/10.1093/qmathj/49.4.433}
{Quart. J. Math. Oxford \textbf{49} (1998), 433-443. }

\bibitem{BT2002}
J.~W.~Bruce and F.~Tari,
Families of surfaces in $\R^4$,
\href{https://doi.org/10.1017/S0013091500000213}
{Proc. Edinb. Math. Soc. \textbf{45} (2002), 181--203.}

\bibitem{David1983}
J.~M.~S.~David, 
Projection-generic curves,
\href{https://doi.org/10.1112/jlms/s2-27.3.552}
{J. London Math. Soc. (2) \textbf{27} (1983), 552--562. }

%\bibitem{DSilva2021}
%J.~L.~Deolindo-Silva, 
%Cross-ratio invariants for surfaces in 4-Space, 
%Bull. Braz. Math. Soc. New Series \textbf{52} (2021), 591--612. 

%\bibitem{DSilva2024}
%J.~L.~Deolindo-Silva, 
%On the differential geometry of smooth ruled surfaces in 4-space, 
%Math. Nachr. \textbf{297} (2024), 4689--4704.

%\bibitem{DSK2019}
%J.~L.~Deolindo-Silva and Y.~Kabata, 
%Projective classification of jets of surfaces in 4-space, 
%Hirosima Math. j. \textbf{49} (2019), 35--46.

%\bibitem{Fomenko1979}
%V.~T.~Fomenko,
%Some properties of two-dimensional surfaces with zero normal torsion in $\mathbb{E}^4$,
%Math. USSR. Sbornik \textbf{35} (1979), 251--265.

\bibitem{FH2012-1}
T.~Fukui and M.~Hasegawa,
Singularities of parallel surfaces, 
\href{https://doi.org/10.2748/tmj/1347369369}
{Tohoku Math. J. \textbf{64} (2012), 387--408.}

\bibitem{GMRFR2000}
R. Garcia, D. K. H. Mochida, M. C. Romero-Fuster and M. A. S. Ruas,
Inflection points and topology of surfaces in 4-space, 
\href{https://doi.org/10.2307/118196}
{Trans. Am. Math. Soc. \textbf{352} (2000), 3029--3043.}

\bibitem{GS2000}
R.~Garcia and J.~Sotomayor, 
Lines of axial curvature on surfaces immersed in $\R^4$, 
\href{https://doi.org/10.1016/S0926-2245(00)00015-2}
{Differ. Geom. Appl. \textbf{12} (2000), 253--269.}

\bibitem{GR1983}
I.~V. Guadalupe and L. Rodriguez,
Normal curvature of surfaces in space forms, 
\href{https://doi.org/10.2140/pjm.1983.106.95}
{Pac. J. Math. \textbf{106} (1983), no.~1, 95--103.}

\bibitem{GGTG1997}
C.~Gutierrez, I.~Guadalupe, R.~Tribuzy and V.~Gu\'i\~nez,
Lines of curvature on surfaces immersed in $\R^4$, 
\href{https://doi.org/10.1007/BF01233393}
{Bol. Soc. Bras. Mat. \textbf{28} (1997), 233-251.}

\bibitem{IFRT2015}
S.~Izumiya, M.~C.~Romero Fuster, M.~A.~S.~Ruas and F.~Tari, Differential Geometry from a Singularity Theory Viewpoint,
\href{https://doi.org/10.1142/9108}
{World Scientific, 2015.}

\bibitem{Little1969}
J.~A. Little,
On singularities of submanifolds of higher dimensional Euclidean spaces, 
\href{https://doi.org/10.1007/BF02411172}
{Ann. Mat. Pura Appl. \textbf{83} (1969), 261--335.}

\bibitem{Maccori2004}
A.~J.~Maccori, 
Superf\'icies Riemannianas no $\R^4$, 
Master's thesis, University of S\~ao Paulo (2004).

\bibitem{MRFR1995}
D.~K.~H.~Mochida, M.~D.~C.~Romero Fuster and M.~A.~S.~Ruas,
The geometry of surfaces in 4-space from a contact viewpoint, 
\href{https://doi.org/10.1007/BF01265348}
{Geom. Dedicata \textbf{54} (1995), 323--332.}

%\bibitem{Montaldi1986-1}
%J.~A.~Montaldi,
%On contact between submanifolds,
%Michigan Math. J. {\bf 33} (1986), 81--85.

\bibitem{Mond1982}
D.~M.~Q.~Mond, 
Classification of certain singularities and applications to differential geometry,
PhD thesis, University of Liverpool (1982).

\bibitem{Mond1985}
D.~M.~Q.~Mond, 
On the classification of germs of maps from $\R^2$ to $\R^3$, 
\href{https://doi.org/10.1112/plms/s3-50.2.333}
{Proc. Lond. Math. Soc. \textbf{50} (1985), 333--369.}

\bibitem{MW1916}
C.~L.~.E.~Moore and E.~B.~Wilson, 
Differential geometry of two dimensional surfaces in hyperspace, 
\href{https://www.jstor.org/stable/20025687}
{Proc. Am. Acad. Arts Sci. \textbf{52} (1916), 269--368.}

\bibitem{NBT2007}
J.~J.~Nu\~no-Ballesteros and F.~Tari, 
Surfaces in $\R^4$ and their projections to 3-spaces, 
\href{https://doi.org/10.1017/S0308210506000758}
{Proc. R. Soc. Edinburgh, Sect. A \textbf{137} (2007), 1313--1328.}

\bibitem{OST2015}
R.~Oset Sinha and F.~Tari, 
Projections of surfaces in $\R^4$ to $\R^3$ and the geometry of their singular images, 
\href{https://doi.org/10.4171/RMI/825}
{Rev. Mat. Iberoam. \textbf{31} (2015), 33--50.}

\bibitem{Porteous1971}
I.~R.~Porteous, 
The normal singularities of a submanifold, 
\href{https://doi.org/10.4310/jdg/1214430015}
{J. Differ. Geom. \textbf{5} (1971), 543--564. }

\bibitem{RFSC2013}
M.~C.~Romero-Fuster and E.~Sanabria-Codesal, 
Conformal invariants and spherical contacts of surfaces in $\R^4$, 
\href{https://doi.org/10.1007/s13163-011-0086-3}
{Rev. Mat. Complut. \textbf{26} (2013), 215--240. }

\bibitem{Tari2009}
F.~Tari, 
Self-adjoint operators on surfaces in $\R^n$, 
\href{https://doi.org/10.1016/j.difgeo.2008.10.010}
{Differ. Geom. Appl. \textbf{27} (2009), 296--306. }

\bibitem{Wong1946}
Y.-C. Wong,
Contributions to the theory of surfaces in 4-space of constant curvature, 
\href{https://doi.org/10.2307/1990268}
{Trans. Amer. Math. Soc. \textbf{59} (1946), 467--507.}
\end{thebibliography}
\end{document}